\documentclass[10pt, leqno, draft]{amsart}

\usepackage{mathrsfs}
\usepackage{amsfonts}
\usepackage{amsthm}
\usepackage{amsmath}
\usepackage{graphicx}
\usepackage{amscd,amssymb,amsthm}
\usepackage{setspace}

\numberwithin{equation}{section}

\title[Integral representations of modified Struve function]{Integral representations and summations of modified Struve function}

\author[\'Arp\'ad Baricz]{\'Arp\'ad Baricz}
\address{Department of Economics, Babe\c{s}-Bolyai University, Cluj-Napoca 400591, Romania} \email{bariczocsi@yahoo.com}

\author{Tibor K. Pog\'any}
\address{Faculty of Maritime Studies, University of Rijeka, 51000 Rijeka, Croatia} \email{poganj@pfri.hr}



\newtheorem{theorem}{Theorem}
\newtheorem{corollary}{Corollary}
\newtheorem{remark}{Remark}

\allowdisplaybreaks

\begin{document}

\maketitle

\begin{abstract} It is known that Struve function $\mathbf H_\nu$ and modified Struve function $\mathbf L_\nu$
are closely connected to Bessel function of the first kind $J_\nu$ and to modified Bessel function of the first
kind $I_\nu$ and possess representations through higher transcendental functions like generalized hypergeometric
${}_1F_2$ and Meijer $G$ function. Also, the NIST project and Wolfram formula collection contain a set of Kapteyn
type series expansions for $\mathbf L_\nu(x)$. In this paper firstly, we obtain various another type integral
representation formulae for $\mathbf L_\nu(x)$ using the technique developed by D. Jankov and the authors. Secondly,
we present some summation results for different kind of Neumann, Kapteyn and Schl\"omilch series built by $I_\nu(x)$
and $\mathbf L_\nu(x)$ which are connected by a Sonin--Gubler formula, and by the associated modified Struve
differential equation. Finally, solving a Fredholm type convolutional integral equation of the first kind,
Bromwich--Wagner line integral expressions are derived for the Bessel function of the first kind $J_\nu$
and for an associated  generalized Schl\"omilch series.
\end{abstract}

\section{\bf Introduction}
\setcounter{equation}{0}

Bessel and modified Bessel function of the first kind, Struve and modified Struve function possess power series
representation of the form \cite{watson}:
   \begin{align*}
      J_\nu(z) &= \sum_{n \ge 0} \dfrac{(-1)^n \left(\frac x2\right)^{2n+\nu}}{\Gamma(n+\nu+1)\, n!}, \qquad \qquad \quad
      I_\nu(z) = \sum_{n \ge 0} \dfrac{ \left(\frac x2\right)^{2n+\nu}}{\Gamma(n+\nu+1)\, n!}\\
      \mathbf H_\nu(z) &= \sum_{n \ge 0} \dfrac{(-1)^n \left(\frac x2\right)^{2n+\nu+1}}
                          {\Gamma\left(n+\frac32\right)\Gamma\left(n+\nu+\frac32\right)}, \qquad
      \mathbf L_\nu(z) = \sum_{n \ge 0} \dfrac{\left(\frac x2\right)^{2n+\nu+1}}
                         {\Gamma\left(n+\frac32\right)\Gamma\left(n+\nu+\frac32\right)}\, ,
   \end{align*}
where $\nu \in \mathbb R$ and $z \in \mathbb C$. Struve \cite{struve} introduced $\mathbf H_\nu$ function as the series
solution of the nonhomogeneous second order Bessel type differential equation (which carries his name). However, the
modified Struve function $\mathbf L_\nu$ appeared into mathematical literature by Nicholson \cite[p. 218]{nicholson}.
Applications of Struve functions are manyfold and include among others optical investigations \cite[pp. 392--395]{walker};
general expression of the power carried by a transverse magnetic or electric beam, is given in terms of
$\mathbf L_{n+\frac12}$ \cite{april}; triplet phase shifts of the scattering by the singular nucleon-nucleon
potentials $\propto \exp(-x)/x^n$ \cite{gajewski}; leakage inductance in transformer windings \cite{hurwil};
boundary element solutions of the two-dimensional multi-energy-group neutron diffusion equation which governs
the neutronic phenomena in nuclear reactors \cite{itagaki}; effective isotropic potential for a pair of dipoles \cite{mercer};
perturbation approximations of lee--waves in a stratified flow \cite{miles}; quantum--statistical distribution functions
of a hard--sphere system \cite{nisteruk}; scattering of plane waves by circular cylinders for the general case of
oblique incidence and for both real and complex values of particle refractive index \cite{step};
aerodynamic sensitivities for subsonic, sonic, and supersonic unsteady, non-planar lifting--surface theory \cite{carson};
stress concentration around broken filaments \cite{fichter} and lift and downwash
distributions of oscillating wings in subsonic and supersonic flow \cite{WRW1, WRW}.

Series of Bessel and/or Struve functions in which summation indices
appear in the order of the considered function and/or twist arguments of the constituting functions, can be unified
in a double lacunary form:
   \[ \mathfrak B_{\ell_1, \ell_2}(z) := \sum_{n \ge 1} \alpha_n \mathscr B_{\ell_1(n)}(\ell_2(n)z),\]
where $ x \mapsto \ell_j(x) = \mu_j + a_jx,$ $j\in\{1,2\},$ $x \in\{0,1,\dots\}$, $z \in \mathbb C$ and $\mathscr B_{\nu}$ is
one of the functions $J_{\nu}, I_{\nu}, \mathbf H_{\nu}$ and $\mathbf L_{\nu}$. The classical theory of the Fourier--Bessel
series of the first type is based on the case when $\mathscr B_{\nu} = J_{\nu}$, see the celebrated monograph by Watson
\cite{watson}. However, varying the coefficients of $\ell_1$ and $\ell_2,$ we get three different cases which have not
only deep roles in describing physical models and have physical interpretations in numerous topics of natural sciences
and technology, but are also of deep mathematical interest, like e.g. zero function series \cite{watson}.
Hence we differ: Neumann series (when $a_1 \neq 0, a_2 = 0$), Kapteyn series (when $a_1 \cdot a_2 \neq 0$) and
Schl\"omilch series (when $a_1 = 0, a_2 \neq 0$). Here, all three series are of the first type (the series' terms
contain only one constituting function $\mathscr B_{\nu}$); the second type series contain product terms of two
(or more) members  - not necessarily different ones - from $J_{\nu}, I_{\nu}, \mathbf H_{\nu}$ and $\mathbf L_{\nu}$.
We also point out that the Neumann series (of the first type) of Bessel function of the second kind $Y_{\nu}$,
modified Bessel function of the second kind $K_{\nu}$ and Hankel functions (Bessel functions of the third kind)
$H^{(1)}_{\nu}, H^{(2)}_{\nu}$ have been studied by Baricz, Jankov and Pog\'any \cite{BJP}, while Neumann series
of the second type were considered by Baricz and Pog\'any in somewhat different purposes in \cite{BP, BP1}; see
also \cite{JPS}. An important role has throughout of this paper the Sonin--Gubler formula which connects modified
Bessel function of the first kind $I_\nu,$ modified Struve function $\mathbf L_\nu$ and a definite integral of the
Bessel function of the first kind $J_\nu$ \cite[p. 424]{gubler} (actually a special case of a Sonin--formula
\cite[p. 434]{watson}):
   \begin{equation} \label{W1}
      \int_0^\infty \dfrac{J_\nu(ax)}{x^2+n^2}\, \dfrac{{\rm d}x}{x^\nu} = \dfrac\pi{2n^{\nu+1}}\left( I_\nu(an)
                - \mathbf L_\nu(an)\right),
   \end{equation}
where $\Re(\nu)>-\frac12,$ $a>0$ and $\Re(n)>0;$ see \cite[p. 426]{watson} for the historical background of \eqref{W1}.

Thus, under extended Neumann series (of Bessel $J_{\nu}$ see \cite{watson}) we mean the following
   \[ \mathfrak N_{\mu,\eta}^{\mathscr B}(x) := \sum_{n \ge 1} \beta_n \mathscr B_{\mu n+\eta}(ax),\]
where $\mathscr B_{\nu}$ is one of the functions $I_{\nu}$ and $\mathbf L_{\nu}$. Integral representation discussions
began very recently with the introductory article by Pog\'any and S\"uli \cite{PS1}, which gives an exhaustive
references list concerning physical applications too; see also \cite{BJP}. In Section 2 we will concentrate to
the Neumann series
   \begin{equation} \label{NS}
      \mathfrak N_{\mu,\eta}(x) := \sum_{n \ge 1} \beta_n I_{\mu n+\eta}(ax)\, .
   \end{equation}
Secondly, Kapteyn series of the first type \cite{WK, WK1, NN} are of the form
   \[ \mathfrak K_{\nu, \mu}^{\,\,\mathscr B}(z) := \sum_{n \ge 1} \alpha_n
                   \mathscr B_{\rho + \mu n}\left((\sigma+\nu n)z\right);\]
more details about Kapteyn and Kapteyn--type series for Bessel function can be found also in \cite{BJP0, BP, DD, TLD}
and the references therein. Here we will consider specific Kapteyn--type series of the following form:
   \begin{equation} \label{K1}
      \mathfrak K_{\nu, \mu}^{\,\,\alpha}(x) := \sum_{n \ge 1} \dfrac{\alpha_n}{n^\mu}\, \left( I_{\nu n}(xn)
                                              - \mathbf L_{\nu n}(xn)\right);
   \end{equation}
this series appear as auxiliary expression in the fourth section of the article. Thanks to Sonin--Gubler formula
\eqref{W1} we give an alternative proof for integral representation of $\mathfrak K_{\nu, \mu}^{\,\,\alpha}(x)$,
see Section 3. Thirdly, under Schl\"omilch series \cite[pp. 155--158]{OXS} (Schl\"omilch considered only cases
$\mu\in \{0,1\}$), we understand the functions series
   \begin{equation}\label{A1}
      \mathfrak S_{\mu,\nu}^{\mathscr B}(z) := \sum_{n\ge 1} \alpha_n\,\mathscr B_\mu\left((\nu+n)z\right).
   \end{equation}
Integral representation are recently obtained for this series in \cite{JP1}, summations are given in \cite{TSVA}.
Our attention is focused currently to
   \[ \mathfrak S_{\mu,\nu}^{I, \mathbf L}(z) := \sum_{n\ge 1} \dfrac{\alpha_n}{n^\mu}\left(I_\nu(xn)
                                               - \mathbf L_\nu(xn)\right)\, .\]
The next generalization is suggested by the theory of Fourier series, and the functions which naturally come
under consideration instead of the classical sine and cosine, are the Bessel functions of the first kind and
Struve's functions. The next type series considered here we call generalized Schl\"omilch series \cite[p. 622]{watson},
\cite[p. 1803]{ito}
   \begin{equation}\label{00}
       \dfrac{a_0}{2\Gamma(\nu+1)} + \left(\dfrac x2\right)^{-\nu} \sum_{n \ge 1} \dfrac{a_nJ_\nu(nx)
              + b_n{\mathbf H}_\nu(nx)}{n^\nu} \, .
   \end{equation}
For further subsequent generalizations consult e.g. Bondarenko's recent article \cite{bondar} and the references therein
and Miller's multidimensional expansion \cite{miller1}. A set of summation formulae of Schl\"omilch series for Bessel
function of the first kind can be found in the literature, such as the Nielsen formula \cite[p. 636]{watson};
further, we have \cite[p. 65]{ECT}, also consult \cite{glasser, PG, MDR, TSVA, zayed1, zayed2}. Similar summations,
for Schl\"omilch series of Struve function, have been given by Miller \cite{miller2}, consult \cite{TSVA} too.

Further, we are interested in a specific variant of generalized Schl\"omilch series in which $J_\nu, \mathbf H_\nu$ are
exchanged by $I_\nu$ and $\mathbf L_\nu$ respectively, when $a_n, b_n$ are of the form
   $ a_0 = 0,$ $a_n = 2^{-\nu}n^{\nu-\mu}\,x^\nu = -b_n,$  $\mu\geq\nu>0,$ which results in
  $$\mathfrak T_{\nu, \mu}^{I, \mathbf L}(x) := \sum_{n \ge 1} \dfrac{I_\nu(nx) - {\mathbf L}_\nu(nx)}{n^\mu}.$$
Its alternating variant $\widetilde{\mathfrak T}_{\nu, \mu}^{I, \mathbf L}(x)$ we perform setting $(-1)^{n-1}a_n \mapsto a_n,$ where $n \in\{0,1,\dots\}$:
   \[ \widetilde{\mathfrak T}_{\nu, \mu}^{I, \mathbf L}(x) := \sum_{n \ge 1}
                 \dfrac{(-1)^{n-1}}{n^\mu}\left(I_\nu(nx) - {\mathbf L}_\nu(nx)\right)\, .\]
Summations of these series are one of tools in obtaining explicit expressions for integrals containing
Butzer--Flocke--Hauss complete Omega--function $\Omega(x)$ \cite{butzer, butzer1, BPS} and Mathieu series
$S(x), \widetilde S(x)$ \cite{mathieu, PST}.

Rayleigh \cite{SR} has showed that series $\mathfrak S_{0,\nu}^B(z)$ play important roles in physics, because they are useful
in investigation of a periodic transverse vibrations uniformly distributed in direction through the two dimensions of the membrane.
Also, Schl\"omilch series present various features of purely mathematical interest and it is remarkable that a null--function
can be represented by such series in which the coefficients are not all zero \cite[p. 622]{watson}. Summation results
in form of a double definite integral representation for $\mathfrak S_{\mu, \nu}^J(z)$, achieved {\em via} Kapteyn--series, have
been recently derived in \cite{JP}.

Finally, we mention that except the Sonin--Gubler formula \eqref{W1} another main tool we refer to is the Cahen formula on the
Laplace integral representation of Dirichlet series. Namely, the Dirichlet series
   \[ \mathscr D_{\boldsymbol a}(r) = \sum_{n \ge 1} a_n e^{-r\lambda_n},\]
where $\Re(r)>0,$ having positive monotone increasing divergent to infinity sequence $(\lambda_n)_{n \ge 1}$, possesses Cahen's integral
representation formula \cite[p. 97]{Cah}
   \begin{equation} \label{W6}
      \mathscr D_{\boldsymbol a}(r) = r \int_0^\infty e^{-rt} \sum_{n \colon \lambda_n \le t} a_n\, {\rm d}t =
                r \int_0^\infty \int_0^{[\lambda^{-1}(t)]} \mathfrak d_u a(u)\, {\rm d}t\,{\rm d}u \, ,
   \end{equation}
where $\mathfrak d_x:=1+\{x\}\frac{\rm d}{{\rm{d}}x}$. Here, $[x]$ and $\{x\}=x-[x]$ denote the integer and fractional part
of $x \in \mathbb{R}$, respectively. Indeed, the so--called counting sum $$\mathscr A_{\boldsymbol a}(t) =
\sum_{n \colon \lambda_n \le t} a_n$$ we find by the Euler--Maclaurin summation formula, following the procedure developed
by the second author \cite{P1}, see also \cite{PST}. Namely
   \[ \mathscr A_{\boldsymbol a}(t) = \sum_{n=1}^{[\lambda^{-1}(t)]} a_n
                                    = \int_0^{[\lambda^{-1}(t)]} \mathfrak d_u a(u)\, {\rm d}u  \, ,\]
since $\lambda \colon \mathbb R_+ \mapsto \mathbb R_+$ is monotone, there exists unique inverse $\lambda^{-1}$ for the
function $\lambda \colon \mathbb R_+ \mapsto \mathbb R_+$, $\lambda|_{\mathbb N} = (\lambda_n)$.

\section{\bf ${\mathbf L}_\nu$ as a Neumann series of modified Bessel $I$ functions}

Let us observe the well--known formulae \cite[Eqs. 11.4.18--19--20]{nist}
   \[ \mathbf H_\nu(z) = \begin{cases}
          \dfrac4{\sqrt{\pi}\,\Gamma\left(\nu+\frac12\right)} \displaystyle \sum_{n \ge 0} \dfrac{(2n+\nu+1)\Gamma(n+\nu+1)}
                  {n!(2n+1)(2n+2\nu+1)}\, J_{2n+\nu+1}(z) &\\ \\
          \sqrt{\dfrac z{2\pi}}\, \displaystyle \sum_{n \ge 0} \dfrac{\left(\frac{z}{2}\right)^n}{n!(n+\frac12)} \,
                  J_{n+\nu+\frac12}(z) &\\ \\
          \dfrac{\left(\frac{z}{2}\right)^{\nu+\frac12}}{\Gamma\left(\nu+\frac12\right)}\,
                  \displaystyle \sum_{n \ge 0} \dfrac{\left(\frac{z}{2}\right)^n}{n!\left(n+\nu+\frac12\right)} \, J_{n+\frac12}(z)
          \end{cases}\, ,\]
where the first formula is valid for $-\nu \not\in \mathbb N$. So, having in mind that $\mathbf L_\nu(z) =
-{\rm i}^{1-\nu} \mathbf H_\nu({\rm i}z)$ and $J_\nu(iz) = {\rm i}^\nu I_\nu(z)$, we immediately conclude that
   \begin{equation} \label{V2}
      \mathbf L_\nu(z) = \begin{cases}
          \dfrac4{\sqrt{\pi}\,\Gamma\left(\nu+\frac12\right)} \displaystyle \sum_{n \ge 0} \dfrac{(-1)^n(2n+\nu+1)\Gamma(n+\nu+1)}
                  {n!(2n+1)(2n+2\nu+1)}\, I_{2n+\nu+1}(z) &\\ \\
          \sqrt{\dfrac z{2\pi}}\, \displaystyle \sum_{n \ge 0} \dfrac{\left(-\frac{z}{2}\right)^n}{n!\left(n+\frac12\right)} \,
                  I_{n+\nu+\frac12}(z) &\\ \\
          \dfrac{\left(\frac{z}{2}\right)^{\nu+\frac12}}{\Gamma\left(\nu+\frac12\right)}\,
                  \displaystyle \sum_{n \ge 0} \dfrac{\left(-\frac{z}{2}\right)^n}{n!\left(n+\nu+\frac12\right)} \, I_{n+\frac12}(z)
          \end{cases}\, .
    \end{equation}
However, all three series expansions we recognize as Neumann--series built by modified Bessel functions of the
first kind. This kind of series have been intensively studied very recently by the authors and D. Jankov in \cite{BJP}. Exploiting
the appropriate findings of that article, we give new integral expressions for the modified Struve function $\mathbf L_\nu$.

First let us modestly generalize \cite[Theorem 2.1]{BJP} which concerns $\mathfrak N_{1,\nu}(x)$, to integral expression for $\mathfrak N_{\mu, \eta}$ defined by \eqref{NS}, following the same procedure as in \cite{BJP}.

\begin{theorem}
Let $\beta \in {\rm C}^1(\mathbb{R}_+),$ $\beta|_{\mathbb{N}}=\{\beta_n\}_{n\in \mathbb{N}}$, $\mu>0$ and assume that
   \begin{equation} \label{V21}
      \lim_{n \to \infty} \dfrac{|\beta_n|^{\frac1{\mu n}}}n < \dfrac\mu{\rm e}\, .
   \end{equation}
Then, for $\mu, \eta$ such that $\min\{\eta+\tfrac32, \mu+\eta+1\}>0$ and
   \[x \in \left(0, 2\, {\rm min}\left(1,\left(\left( e/\mu\right)^\mu\;\limsup\limits_{n\to \infty}n^{-\mu}
                   |\beta_n|^{1/n}\right)^{-1}\right)\right) := \mathscr I_\beta\,,\]
we have the integral representation
   \begin{equation} \label{V3}
       \mathfrak N_{\mu,\eta}(x) = - \int_1^{\infty} \int_0^{[u]}\dfrac{\partial}{\partial u}
                     \left( \Gamma\left(\mu u+\eta + \frac12\right)\,I_{\mu u + \eta}(x)\right)
                     \mathfrak d_s\left( \frac{\beta(s)}{\Gamma\left(\mu s + \eta + \frac12\right)}\right) {\rm d}u{\rm d}s\,.
   \end{equation}
\end{theorem}

\begin{proof} The proof is a copy of the proving procedure delivered for \cite[Theorem 2.1]{BJP}. The only exception is to refine the
convergence condition upon $\mathfrak N_{\mu,\eta}(x)$. By the bound \cite[p. 583]{B3}:
   \[ I_\nu (x) < \frac{\left( \frac x2\right)^\nu}{\Gamma(\nu+1)} \, {e}^{\tfrac{x^2}{4(\nu+1)}},\]
where $x>0$ and $\nu+1>0,$ we have
   \[ \left|\mathfrak N_{\mu, \eta}(x)\right| < \left(\dfrac x2\right)^{\mu+\eta} \,{e}^{\tfrac{x^2}{4(\mu + \eta+1)}}
               \sum_{n \ge 1} \dfrac{\left|\beta_n\right|}{\Gamma(\mu n + \eta + 1)} \, , \]
so, the absolute convergence of the right hand side series suffices for the finiteness of $\mathfrak N_{\mu, \eta}(x)$ on
$\mathscr I_\beta$. However, condition \eqref{V21} ensures the absolute convergence by the Cauchy convergence criterion.

The remaining part of the proof mimicks the one performed for \cite[Theorem 2.1]{BJP}, having in mind that
$\mu =1$ reduces Theorem 2 to the ancestor result \cite[Theorem 2.1]{BJP}.
\end{proof}

\begin{theorem}
If $\nu>0$ and $x \in (0, 2),$ then we have the integral representation
   \begin{align} \label{V4}
       \mathbf L_\nu(x) &- \dfrac{2\Gamma(\nu+2)}{\sqrt{\pi}\,\Gamma\left(\nu+\frac32\right)}\, I_{\nu+1}(x)\nonumber \\
                        &= \int_1^{\infty} \int_0^{[u]}\,\dfrac{\partial}{\partial u}
                           \left( \Gamma\left(2 u+\nu + \tfrac32\right)\,I_{2 u + \nu+1}(x)\right)
                           \mathfrak d_s\left( \frac{\beta(s)}{\Gamma\left(2 s + \nu
                         + \frac32\right)}\right) {\rm d}u\,{\rm d}s\,,
   \end{align}
where
   \[ \beta(s) = - \dfrac{e^{{\rm i}\pi s}(2s + \nu + 1)\Gamma(s+\nu+1)}{\sqrt{\pi}\, \Gamma\left(\nu+\frac12\right)\,
                 \Gamma(s+1)\, \left(s+\frac12\right)\left(s+\nu+\frac12\right)}\, .\]
\end{theorem}

\begin{proof} Consider the first Neumann sum expansion of $\mathbf L_\nu(x)$ in \eqref{V2}, that is
   \begin{align*}
      \mathbf L_\nu(x) &= \dfrac4{\sqrt{\pi}\,\Gamma\left(\nu+\frac12\right)} \sum_{n \ge 0}
                          \dfrac{(-1)^n(2n+\nu+1)\Gamma(n+\nu+1)}{n!(2n+1)(2n+2\nu+1)}\, I_{2n+\nu+1}(x)\\
                       &= \dfrac{2\Gamma(\nu+2)}{\sqrt{\pi}\,\Gamma\left(\nu+\frac32\right)}\, I_{\nu+1}(x)\\
                       &\qquad - \dfrac4{\sqrt{\pi}\,\Gamma\left(\nu+\frac12\right)} \sum_{n \ge 1}
                          \dfrac{(-1)^{n-1}(2n+\nu+1)\Gamma(n+\nu+1)}{n!(2n+1)(2n+2\nu+1)}\, I_{2n+\nu+1}(x) \, .
   \end{align*}
Observe that
   \[ \mathbf L_\nu(x) = \dfrac{2\Gamma(\nu+2)}{\sqrt{\pi}\,\Gamma\left(\nu+\frac32\right)}\, I_{\nu+1}(x) - \mathfrak N_{2, \nu+1}(x)\]
in which we specify
   \[ \beta_n = \dfrac{(-1)^{n-1}(2n + \nu + 1)\Gamma(n+\nu+1)}{\sqrt{\pi}\, \Gamma\left(\nu+\frac12\right)\,\Gamma(n+1)\,
                \left(n+\frac12\right)\left(n+\nu+\frac12\right)}\, .\]
Since
   \[ |\beta(s)| \sim \dfrac{2s^{\nu-2}}{\sqrt{\pi}\, \Gamma\left(\nu + \frac12\right)}, \qquad s \to \infty\, ,\]
we deduce (by means of Theorem 1) that  \eqref{V4} is valid for $x\in\mathscr I_\beta = (0,2)$.
\end{proof}

\begin{theorem} For $\nu+2>0$ and $x \in (0, 2)$ we have the integral representation
   \begin{align} \label{V5}
       \mathbf L_\nu(x) &- \sqrt{\dfrac{2x}\pi}\,I_{\nu+\frac12}(x)\\
                        &= \int_1^{\infty} \int_0^{[u]}\,\dfrac{\partial}{\partial u}
                           \left( \Gamma(u+\nu + 1)\,I_{u + \nu+\frac12}(x)\right)
                           \mathfrak d_s\left( \frac{\beta(s)}{\Gamma(s + \nu + 1)}\right) {\rm d}u\,{\rm d}s\,,
   \end{align}
where
   \[ \beta(s) = - \sqrt{ \dfrac x{2\pi}}\,\dfrac{e^{{\rm i}\pi s}\left(\frac x2\right)^s}{\Gamma(s+1)\, (s+\frac12)}\, .\]
\end{theorem}

\begin{proof} Let us observe now the second Neumann sum expansion of $\mathbf L_\nu(x)$ in \eqref{V2}:
   \begin{align*}
      \mathbf L_\nu(x) &= \sqrt{\dfrac x{2\pi}}\, \displaystyle \sum_{n \ge 0}
                          \dfrac{\left(-\frac x2\right)^n}{n!\left(n+\frac12\right)} \, I_{n+\nu+\frac12}(x)\\
                       &= \sqrt{\dfrac{2x}\pi}\,I_{\nu+\frac12}(x)
                        - \sqrt{\dfrac x{2\pi}}\, \sum_{n \ge 1}
                          \dfrac{(-1)^{n-1}\left(\frac x2\right)^n}{n!\left(n+\frac12\right)}\, I_{n+\nu+\frac12}(x) \, .
   \end{align*}
In other words,
   \[ \mathbf L_\nu(x) = \sqrt{\dfrac{2x}\pi}\,I_{\nu+\frac12}(x) - \mathfrak N_{1, \nu+\frac12}(x)\]
in which we specify
   \[ \beta(s) = -\sqrt{ \dfrac x{2\pi}}\, \dfrac{e^{{\rm i}\pi s}\left( \frac x2\right)^s}{\Gamma(s+1)\left(s+\frac12\right)}\, .\]
The convergence condition \eqref{V21} reduces to the behavior of the auxiliary series
   \[ \sum_{n \ge 0} \dfrac{|\beta_n|}{\Gamma\left(n + \nu+\frac12\right)} \sim \sqrt{ \dfrac{2x}\pi}\,
      {}_1F_2 \left( \begin{array}{c} \frac12 \\ \frac32,\, \nu+\frac12 \end{array} \left| \dfrac{|x|}2 \right) \right. \, ,\]
which is convergent for all bounded $x \in \mathbb C$, unconditionally upon $\nu.$ Here ${}_1F_2$ denotes the
hypergeometric function defined by series \cite[p. 62]{AAR}
   \[ {}_1F_2 \left( \left.\begin{array}{c} a \\ b_1,\, b_2 \end{array} \right| z \right) =
      \sum_{n \ge 0} \dfrac{(a)_n}{(b_1)_n(b_2)_n}\, \dfrac{z^n}{n!}\, .\]
However, for $\nu>-2$ we have the integral expression \cite[p. 79]{watson}
   \begin{equation} \label{V6}
      I_\nu(z) = \dfrac{2^{1-\nu}z^\nu}{\sqrt{\pi}\, \Gamma\left(\nu+\frac12\right)}\int_0^1(1-t^2)^{\nu-\frac12}\cosh(zt){\rm d}t,
   \end{equation}
where $ z\in\mathbb{C}$ and $\Re(\nu)>-\frac12.$ This was used in the proof of the ancestor result \eqref{V3}, see \cite[Theorem 2.1]{BJP}. Now, we apply Theorem 1 and conclude that \eqref{V5} is valid for $x\in\mathscr I_\beta = (0,2)$.
\end{proof}

The third formula in \eqref{V2} one reduces to the case $\mathfrak N_{1, \frac12}(x)$. However, we shall omit the proof, since the slightly repeating derivation procedure used for \eqref{V5} directly gets the desired integral expression.

\begin{theorem} If $\nu+2>0$ and $x \in (0, 2),$ then we have the integral representation
   \begin{align*}
       \mathbf L_\nu(x) - \dfrac{x^\nu\, \sinh x}{2^\nu\sqrt{\pi}\,\Gamma\left(\nu+\frac32\right)} =
                        \int_1^{\infty} \int_0^{[u]}\,\dfrac{\partial}{\partial u}
                           \left( \Gamma\left(u+1\right)\,I_{u + \frac12}(x)\right)
                           \mathfrak d_s\left( \frac{\beta(s)}{\Gamma\left(s + 1\right)}\right) {\rm d}u\,{\rm d}s\,,
   \end{align*}
where
   \[ \beta(s) = - \dfrac{(\frac x2)^{\nu+\frac12}}{\Gamma\left(\nu+\frac12\right)}\,
                 \dfrac{e^{{\rm i}\pi s}\left( \frac x2\right)^s}{\Gamma(s+1)\left(s + \nu + \frac12\right)}\, .\]
\end{theorem}

Now, applying the integral representation \eqref{V6} we derive another integral expression for $\mathbf L_\nu(x)$ in terms of
hypergeometric functions in the integrand.

\begin{theorem} Let $\nu>-\tfrac12$. Then for $x>0$ we have
   \begin{align*}
      \mathbf L_\nu(x) &= \dfrac{x^{\nu+1}\Gamma(\nu+2)}{\sqrt{\pi}\,2^{2\nu-\frac12} \Gamma\left(\nu + \frac32\right)
                          \Gamma\left( \frac\nu2+\frac34\right)\Gamma\left(\frac\nu2+\frac54\right)} \int_0^1 (1-t^2)^{\nu+\frac12}\cosh(xt)\,
                          \nonumber \\
                       &\qquad \times {}_4F_5 \left(\left. \begin{array}{c}
                          \frac12,\,\frac{\nu+3}2,\, \nu+\frac12,\, \nu+1 \\ \\ \frac32, \frac{\nu+1}2, \frac\nu2+\frac34,
                          \frac\nu2+\frac54, \nu+\frac34 \end{array} \right| - \frac{x^2}{16}(1-t^2)^2\right)\, {\rm d}t\, ,
   \end{align*}
where
   \[ {}_4F_5 \left(\left.\begin{array}{c} a_1,\,a_2,\,a_3,\, a_4 \\ b_1, b_2, b_3, b_4, b_5 \end{array} \right|  z \right)
         = \sum_{n \ge 0} \dfrac{\prod\limits_{j=1}^4(a_j)_n}{\prod\limits_{j=1}^5(b_j)_n}\, \dfrac{z^n}{n!}\,\]
stands for the generalized hypergeometric function with four upper and five lower parameters.
\end{theorem}

\begin{proof} Consider the first Bessel function series expansion for $\mathbf L_\nu(x)$ given in \eqref{V2}. Applying
{\it mutatis mutandis} the integral representation formula \eqref{V6}, the Pochhammer symbol technique, the familiar formula ${(A)_n}(n+A) = A{(A+1)_n},$ $n\in\{0,1,\dots\},$ and the duplication formula
   \[ \Gamma(2z) = \dfrac{2^{2z-1}}{\sqrt{\pi}}\, \Gamma(z)\Gamma\left(z+\tfrac12\right) \]
to the summands, we get the chain of equivalent legitimate transformations:
   \begin{align*}
      \mathbf L_\nu(x) &= \dfrac8{\sqrt{\pi}\, \Gamma\left(\nu+\frac12\right)} \sum_{n \ge 0} \dfrac{(-1)^n (2n+\nu+1)\Gamma(n+\nu+1)}
              {n!(2n+1)(2n+2\nu+1)}\dfrac{2\left(\frac x2\right)^{2n+\nu+1}}{\sqrt{\pi}\,\Gamma\left(2n+\nu+\frac32\right)}\\
           &\qquad \times \int_0^1 (1-t^2)^{2n+\nu+\frac12} \cosh(xt)\, {\rm d}t \\
           &= \dfrac{4\left(\frac x2\right)^{\nu+1}}{\pi \Gamma\left(\nu + \frac12\right)}
              \int_0^1 (1-t^2)^{\nu +\frac12} \cosh(xt)\\
           &\qquad \times \sum_{n \ge 0}
              \dfrac{\left(n + \frac{\nu+1}2\right) \Gamma(n + \nu +1)\, \left[ - \frac{x^2}4\,(1-t^2)^2\right]^n}
              {\left(n + \frac12\right)\left(n + \nu + \frac12\right)\Gamma\left(2n + \nu+\frac32\right)\, n!}\, {\rm d}t\\
           &= \dfrac{4\left( \frac x2\right)^{\nu+1}(\nu+1)\Gamma(\nu+1)}{\sqrt{\pi}\, \left(\nu+\frac12\right)\Gamma\left(\nu
            + \frac12\right) \Gamma\left(\frac\nu2  + \frac34\right)\Gamma\left(\frac\nu2+\frac54\right)} \,
              \int_0^1 (1-t^2)^{\nu +\frac12} \cosh(xt)\\
           &\qquad \times \sum_{n \ge 0}
              \dfrac{(\frac12)_n(\frac{\nu+3}2)_n(\nu+\frac12)_n(\nu+1)_n\,\left[ - \frac{x^2}{16}\,(1-t^2)^2\right]^n}
              {{\left(\frac32\right)}_n{\left(\frac{\nu+1}2\right)}_n{\left(\frac\nu2+\frac34\right)}_n
              {\left(\frac\nu2+\frac54\right)}_n{\left(\nu+\frac34\right)}_n\, n!}\, {\rm d}t,
   \end{align*}
which proves the assertion.
\end{proof}

By virtue of similar manipulations presented above, we conclude the following results.

\begin{theorem} Let $\nu>-\tfrac12$ and $x>0$. Then there holds
   \[ \mathbf L_\nu(x) = \begin{cases}
           \dfrac{x^{\nu+1}}{2^{\nu-1}\pi\,\Gamma(\nu+1)} \displaystyle \int_0^1 (1-t^2)^{\nu}\cosh(xt)\,
                          {}_1F_2 \left(\left. \begin{array}{c}
                          \frac12 \\  \frac32,\, \nu+1 \end{array} \right| - \frac{x^2}4(1-t^2)\right)\, {\rm d}t \\
           \dfrac{x^{\nu+1}}{\sqrt{\pi}\,2^{\nu+\frac12}\,\Gamma\left(\nu+\frac12\right)} \displaystyle \int_0^1 \cosh(xt)\,
                          {}_1F_2 \left(\left. \begin{array}{c} \nu+\frac12 \\ 1,\, \nu+\frac32 \end{array} \right| - \frac{x^2}4(1-t^2)
                          \right)\, {\rm d}t
                \end{cases}\, .\]
\end{theorem}

The proof of Theorem 6 follows from the same proving procedure as the previous theorem but now considering the second and third series
expansion results in \eqref{V2}, so we shall omit the proofs of these integral representations.

\section{\bf Integrals containing $\Omega(x)$--function and Mathieu series \textit{\textbf{via}} summation of
$\mathfrak T_{\,\,\,\nu}^{I, \mathbf L}(x)$ }

By virtue of the Sonin--Gubler formula \eqref{W1} we establish the convergence conditions for the generalized Schl\"omilch series
$\mathfrak T_{\nu, \mu}^{I, \mathbf L}(x)$ and $\widetilde{\mathfrak T}_{\nu, \mu}^{I, \mathbf L}(x)$. As for $n$ enough large we have
   \begin{equation} \label{W0}
      I_\nu(an) - \mathbf L_\nu(an) = \dfrac{2n^{\nu-1}}\pi\, \int_0^\infty \dfrac{J_\nu(ax)\, {\rm d}x}{(1+n^{-2}x^2)x^\nu}
                = \mathscr O\left(n^{\nu-1}\right)\, ,
   \end{equation}
we immediately conclude that the following equiconvergences hold true
   \[ \mathfrak T_{\nu, \mu}^{I, \mathbf L}(x) \sim \zeta(\mu-\nu+1),\qquad
      \widetilde{\mathfrak T}_{\nu, \mu}^{I, \mathbf L}(x) \sim \eta(\mu-\nu+1)\, ,\]
that is, $\mathfrak T_{\nu, \mu}^{I, \mathbf L}(x)$ converges for $\mu>\nu>0$, while $\widetilde{\mathfrak T}_{\nu, \mu}^{I,
\mathbf L}(x)$ converges for $\mu +1 > \nu>0$. On the other hand, we connect $\widetilde{\mathfrak T}_{\nu, \nu}^{I, \mathbf L}(x)$ and the Butzer--Flocke--Hauss (BFL) complete
Omega function \cite[Definition 7.1]{butzer}
   \[ \Omega(w) = 2 \int_{0+}^{\frac{1}{2}} \sinh(wu)\cot(\pi u)\,{\rm d}u, \qquad  w\in \mathbb C\,.\]
By the Hilbert transform terminology, $\Omega(w)$ is the Hilbert transform $\mathscr H(e^{-wx})_1(0)$ at $0$ of the $1$-periodic function $(e^{-wx})_1$ defined by the periodic continuation of the following exponential function \cite[p. 67]{butzer}: $e^{-xw},$ $x\in \left[-\tfrac12,\tfrac12\right),$ $w\in \mathbb C,$ that is,
   \[ \mathscr H(e^{-xw})_1(0) := \text{P.V.} \int_{-\frac{1}{2}}^{\frac{1}{2}}e^{wu}\cot(\pi u) \, {\rm d}u \equiv \Omega(w) \]
where the integral is to be understood in the sense of Cauchy's  Principal Value at zero, see e.g. \cite{BPS, PSri}.

On the other side by differentiating once \eqref{W1} with respect to $n$ we get a tool to obtain Mathieu series
$S(x)$ (introduced by \'Emile Leonard Mathieu \cite{mathieu}) and its alternating variant $\widetilde S(x)$
(introduced by Pog\'any, Srivastava and Tomovski \cite{PST}), which are defined as follows
   \[ S(x) = \sum_{n \ge 1} \dfrac{2n}{(x^2+n^2)^2}\,, \quad
      \widetilde S(x) = \sum_{n \ge 1} \dfrac{2(-1)^{n-1}n}{(x^2+n^2)^2}.\]
Closed integral expression for $S(r)$ was considered by Emersleben \cite{Emers} and subsequently by Elbert \cite{elbert}, while for
$\widetilde S_\mu(x)$ integral representation has been given by Pog\'any, Srivastava and Tomovski \cite{PST}:
   \begin{align} \label{W4}
      S(x) &= \dfrac1x \int_0^\infty \dfrac{t\, \sin(xt)}{e^t-1}\,{\rm d}t, \\ \label{W5}
      \widetilde S(x) &= \dfrac1x \int_0^\infty \dfrac{t\, \sin(xt)}{e^t+1}\,{\rm d}t \, .
   \end{align}
Another kind integral expressions for underlying Mathieu series can be found in \cite{PST}.

\begin{theorem} Assume that $\Re(\nu)>0$ and $a>0$. Then we have
   \[ \int_0^\infty \dfrac{J_\nu(ax)\,\Omega(2\pi x)}{ x^\nu\,\sinh(\pi x)}\,{\rm d}x
            = \nu \int_0^\infty e^{-\nu t}\, \int_0^{[e^t]} \mathfrak d_u\left( e^{{\rm i}\pi u}\left(\mathbf L_\nu(au)
            - I_\nu(au)\right)\right)\, {\rm d}t\,{\rm d}u\, .\]
\end{theorem}

\begin{proof} When we multiply \eqref{W1} by $(-1)^{n-1}n$ and sum up all three series with respect to $n \in \mathbb N$, the
following partial-fraction representation of the Omega function \cite[Theorem 1.3]{butzer}
   \[ \frac{\pi \Omega(2\pi w)}{\sinh(\pi w)} = \sum_{n=1}^\infty \frac{2(-1)^{n-1} \,n}{n^2+w^2}\, .\]
immediately give
   \[ \int_0^\infty \dfrac{J_\nu(ax)\,\Omega(2\pi x)}{ x^\nu\,\sinh(\pi x)}\,{\rm d}x
            = \sum_{n \ge 1} (-1)^{n-1} n^{-\nu}\,\left( I_\nu(an) - \mathbf L_\nu(an)\right)
            = \widetilde{\mathfrak T}_{\nu, \nu}^{I, \mathbf L}(a) \, .\]
We recognize the right--hand--side sums as Dirichlet series of $I_\nu$ and $\mathbf L_\nu,$ respectively. Being
   \[ \sum_{n \ge 1} (-1)^{n-1} n^{-\nu}I_\nu(an) = \sum_{n \ge 1} e^{{\rm i}\pi(n-1)}I_\nu(an) e^{-\nu \ln n}, \qquad \Re(\nu)>0\, , \]
we get
   \[ \sum_{n \ge 1} (-1)^{n-1} n^{-\nu}I_\nu(an) = \nu \int_0^\infty e^{-\nu t}
                     \sum_{n \colon \ln n \le t} e^{{\rm i}\pi(n-1)}I_\nu(an)\,{\rm d}t .\]
So, making use of the Euler--Maclaurin summation \eqref{W6} to the Cahen's formula we deduce
   \[ \sum_{n \ge 1} (-1)^{n-1} n^{-\nu}I_\nu(an) =
      - \nu\,  \int_0^\infty \int_0^{[e^t]}\,e^{-\nu t}\,  \mathfrak d_u\left( e^{{\rm i}\pi u}I_\nu(au) \right)\,{\rm d}t\,{\rm d}u ;\]
and repeating the procedure to the second Dirichlet series containing $\mathbf L_\nu(an),$ the proof is complete.
\end{proof}

The next result concerns a hypergeometric integral, which one we integrate by means of Schl\"omilch series of
modified Bessel and modified Struve functions.

\begin{theorem} Let $\Re(\nu)>0$ and $a>0$. Then we have
   \begin{align*}
      \int_0^\infty J_\nu(ax)\,S(x) \dfrac{{\rm d}x}{ x^\nu} &=
             \dfrac{\sqrt{\pi} a^{\nu+2}}{2^{\nu+1} \Gamma(\nu+\frac12)} \int_0^1 \dfrac{t^2}{e^{at}-1}\,
             {}_2F_1 \left(\left. \begin{array}{c} \frac12, \frac12-\nu\\ \frac32 \end{array} \right| t^2 \right)\,{\rm d}t\\
          &\qquad + \dfrac{\pi a^\nu\left[ {\rm Li}_2(e^{-a}) + a\,{\rm Li}_1(e^{-a})\right]}{2^{\nu+1}\Gamma(\nu+1)},
   \end{align*}
where ${\rm Li}_\alpha(z)$ stands for the {\em dilogarithm function}.
\end{theorem}

\begin{proof}
Differentiating \eqref{W1} with respect to $n$, we get
   \[ \int_0^\infty \dfrac{2n\, J_\nu(ax)}{(x^2+n^2)^2}\, \dfrac{{\rm d}x}{x^{\nu}} =
             \dfrac{\pi(\nu+1)}{2 n^{\nu+2}} \left( I_\nu(an) - \mathbf L_\nu(an)\right)
           - \dfrac{a\pi}{2n^{\nu+1}}\left( I_\nu'(an) - \mathbf L_\nu'(an)\right)\, .\]
Summing up this relation with respect to positive integers $n \in \mathbb N$, we get
   \begin{align*}
      \kappa_\nu(a) &:= \int_0^\infty J_\nu(ax)\,S(x) \dfrac{{\rm d}x}{ x^\nu} =
             \dfrac{\pi(\nu+1)}2 \sum_{n \ge 1} n^{-\nu-2} \left( I_\nu(an) - \mathbf L_\nu(an)\right) \nonumber \\
           &\,\, - \dfrac{a\pi}2 \sum_{n \ge 1} n^{-\nu-1}\left( I_\nu'(an) - \mathbf L_\nu'(an)\right)
            = \dfrac{\pi(\nu+1)}2\,{\mathfrak T}_{\nu, \nu+2}^{I, \mathbf L}(a)
            - \dfrac{a\pi}2\, \dfrac{\rm d}{{\rm d}a}\,{\mathfrak T}_{\nu, \nu+2}^{I, \mathbf L}(a)\,.
   \end{align*}
Now, by the Emersleben--Elbert formula \eqref{W4} we conclude that
   \[ \kappa_\nu(a) = \int_0^\infty J_\nu(ax)\,S(x) \dfrac{{\rm d}x}{ x^\nu}
                    = \int_0^\infty \dfrac{t}{e^t-1} \left( \int_0^\infty \dfrac{J_\nu(ax)\,\sin(xt)}
                      {x^{\nu+1}}\, {\rm d}x\right)\,{\rm d}t\, . \]
Expressing the sine {\it via} $J_{\frac12}$, we get that the inner--most integral equals
   \begin{equation} \label{W81}
      \int_0^\infty \dfrac{J_\nu(ax)\, \sin(xt)}{x^{\nu+1}}\, {\rm d}x = \sqrt{\dfrac{\pi t}2} \int_0^\infty
             \dfrac{J_\nu(ax)J_{\frac12}(tx)}{x^{\nu+\frac12}}\, {\rm d}x \, .
   \end{equation}
Now, we shall apply the Weber--Sonin--Schafheitlin formula \cite[\S {\bf 13.41}]{watson} for $\lambda = \nu + \frac12$,
which reduces to
    \[ \int_0^\infty J_\nu(ax)J_{\frac12}(tx)x^{-\nu-\frac12}\, {\rm d}x
             = \begin{cases}
      \dfrac{a^\nu \sqrt{\pi}}{2^{\nu+\frac12}\, \sqrt{t}\,\Gamma(\nu+1)}, & \quad \mbox{if}\ \ \ 0<a \le t \\
      \dfrac{a^{\nu-1}\,\sqrt{t}}{2^{\nu+\frac12}\Gamma\left(\nu+\frac12\right)}
            \, {}_2F_1 \left(\left. \begin{array}{c} \frac12, \frac12-\nu\\ \frac32 \end{array} \right| \dfrac{t^2}{a^2}\right),
            & \quad \mbox{if}\ \ \ 0<t<a  \end{cases}\, .\]
Accordingly, \eqref{W81} becomes
   \begin{align*}
      \kappa_\nu(a) &= \dfrac{\sqrt{\pi} a^{\nu-1}}{2^{\nu+1} \Gamma(\nu+\frac12)} \int_0^a \dfrac{t^2}{e^t-1}\,
             {}_2F_1 \left(\left. \begin{array}{c} \frac12, \frac12-\nu\\ \frac32 \end{array} \right| \dfrac{t^2}{a^2}\right)\,{\rm d}t
           + \dfrac{\pi a^\nu}{2^{\nu+1}\Gamma(\nu+1)} \int_a^\infty \dfrac{t}{e^t-1}\, {\rm d}t  \nonumber \\\
          &= \dfrac{\sqrt{\pi} a^{\nu+2}}{2^{\nu+1} \Gamma\left(\nu+\frac12\right)} \int_0^1 \dfrac{t^2}{e^{at}-1}\,
             {}_2F_1 \left(\left. \begin{array}{c} \frac12, \frac12-\nu\\ \frac32 \end{array} \right| t^2 \right)\,{\rm d}t
           + \dfrac{\pi a^\nu}{2^{\nu+1}\Gamma(\nu+1)} \int_0^\infty \dfrac{t+a}{e^{t+a}-1}\, {\rm d}t \nonumber \\
          &= \dfrac{\sqrt{\pi} a^{\nu+2}}{2^{\nu+1} \Gamma\left(\nu+\frac12\right)} \int_0^1 \dfrac{t^2}{e^{at}-1}\,
             {}_2F_1 \left(\left. \begin{array}{c} \frac12, \frac12-\nu\\ \frac32 \end{array} \right| t^2 \right)\,{\rm d}t \nonumber \\
          &\qquad + \dfrac{\pi a^\nu}{2^{\nu+1}\Gamma(\nu+1)}\, \left[ {\rm Li}_2(e^{-a}) + a\,{\rm Li}_1(e^{-a})\right]\, ,
   \end{align*}
where the dilogarithm ${\rm Li}_\alpha(z) = \sum_{n \ge 1}z^n n^{-\alpha},$ $|z|\le 1$, has the integral representation
   \[ {\rm Li}_\alpha(z) = \dfrac{z}{\Gamma(\alpha)}\, \int_0^\infty \dfrac{t^{\alpha-1}}{e^t-z}\, {\rm d}t\,, \qquad \Re(\alpha)>0 .\]
This completes the proof.
\end{proof}

\section{\bf Differential equations for Kapteyn and Schl\"omilch series of modified Bessel and modified Struve functions}

Kapteyn series of Bessel functions were introduced by Willem Kapteyn \cite{WK, WK1}, and were considered and discussed
in details by Nielsen \cite{NN} and Watson \cite{watson}, who devoted a whole section of his celebrated monograph to
this theme. Recently, the present authors and Jankov obtained integral representation and ordinary differential
equations descriptions and related results for real variable Kapteyn series \cite{BJP0, JP1}.

Now, we will consider the Kapteyn series built by modified Bessel functions of the first kind, and modified Struve functions
   \[ \mathfrak K_{\nu, \mu}^{\,\,\alpha}(x)  = \sum_{n \ge 1} \dfrac{\alpha_n}{n^\mu}\left( I_{\nu n}(xn)
                                              - \mathbf L_{\nu n}(xn)\right)\,,\]
where the parameter space includes positive $a>0$, while sequence $(\alpha_n)_{n \ge 1}$ ensures the convergence of
$\mathfrak K_{\nu, \mu}^{\,\,\alpha}(x)$. Our first goal is to establish double definite integral representation formula for
$\mathfrak K_{\nu, \mu}^{\,\,\alpha}(x)$. In this goal we recall the definition of the confluent Fox-Wright generalized
hypergeometric function ${}_1\Psi_1^*$ (for the general case ${}_p\Psi_q^*$ consult \cite[p. 493]{SJPS}):
   \begin{equation} \label{FW}
      {}_1\Psi_1^* \left[\left. \begin{array}{c} (a,\rho)\\ (b,\sigma)\end{array} \right| z \right]=
          \sum_{n=0}^\infty \dfrac{ (a)_{\rho n}}{(b)_{\sigma n}} \; \dfrac{z^n}{n!}\, ,
   \end{equation}
where $a, b \in \mathbb C,$ $\rho, \sigma>0$ and where, as usual, $(\lambda)_{\mu}$ denotes the Pochhammer symbol defined, in terms
of Euler's Gamma function, by
   \[ (\lambda)_{\mu} := \dfrac{\Gamma(\lambda+\mu)}{\Gamma(\lambda)}
                       = \begin{cases}
                            1, & \mbox{if}\ \ \ \mu = 0;\, \lambda \in \mathbb C \setminus \{0\}\\
                            \lambda(\lambda+1) \cdots (\lambda + n-1), & \mbox{if}\ \ \ \mu=n \in \mathbb N;\; \lambda \in \mathbb C
                         \end{cases}.\]
The defining series in \eqref{FW} converges in the whole complex $z$-plane when
$\Delta = \sigma-\rho+1>0$; if $\Delta = 0$, then the series converges for $|z|<\nabla$, where $\nabla := \rho^{-\rho}\,\sigma^\sigma$.

\begin{theorem} Let $\mu>\nu>0$ and let $\alpha \in {\rm C}^2(\mathbb R_+)$, such that $\alpha\left|_{\mathbb N} = (\alpha_n)\right.$.
Then for
   \[ x \in \left( 0, 2\min\left\{1, \dfrac{\nu}{{\rm e}\,
                        \limsup\limits_{n \to \infty}|\alpha_n|^{1/\nu n}}\right\} \right) := \mathscr I_\alpha\, , \]
we have
   \begin{equation} \label{M1}
      \mathfrak K_{\nu, \mu}^{\,\,\alpha}(x) = - \int_1^\infty \int_0^{[t]} \dfrac{\partial}{\partial t}\,
                   \dfrac{\Gamma(\nu t+\frac12)}{\Gamma(\nu t)}\left( \dfrac x2\right)^{\nu t}\,
                   {}_1\Psi_1^\star\left[\left. \begin{array}{c} \left( \frac12, \frac12\right) \\ (\nu t, 1) \end{array} \right|
                   -xt \right]\cdot \mathfrak d_s \dfrac{\alpha(s)s^{\nu s-\mu}}{\Gamma\left(\nu s+\frac12\right)}
                   \,{\rm d}t {\rm d}s\, .
   \end{equation}
\end{theorem}

\begin{proof} The Sonin--Gubler formula enables us to transform the summands of the Kapteyn series
$\mathfrak K_{\nu, \mu}^{\,\,\alpha}(x)$
into
   \[ \mathfrak K_{\nu, \mu+1}^{\,\,\alpha}(x) = \frac2\pi \int_0^\infty \sum_{n \ge 1} \dfrac{\alpha_n}{n^{\mu-\nu n}}\,
                   \dfrac{J_{\nu n}(x y)}{(y^2 + n^2)\,y^{\nu n}}\, {\rm d}y \, .\]
Making use of the Gegenbauer's integral expression for $J_\alpha$ \cite[p. 204, Eq. (4.7.5)]{AAR}, after some algebra we get
   \begin{align*}
      \mathfrak K_{\nu, \mu+1}^{\,\,\alpha}(x) &= \dfrac1{\sqrt{\pi}} \int_0^1\dfrac1{\sqrt{1-t^2}}\,
                  \left\{ \sum_{n \ge 1} \dfrac{\alpha_n\, \left( \frac x2(1-t^2)\right)^{\nu n}}
                  {n^{\mu-\nu n}\, \Gamma\left( \nu n + \frac12\right)} \int_0^\infty \dfrac{\cos(xty)}{y^2+n^2}\,{\rm d}y\right\}\, {\rm d}t\\
               &= \dfrac2{\sqrt{\pi}} \int_0^1\dfrac1{\sqrt{1-t^2}}\,
                  \left\{ \sum_{n \ge 1} \dfrac{\alpha_n\, \left(\frac x2(1-t^2)\right)^{\nu n}{\rm e}^{-xtn}}
                  {n^{\mu-\nu n+1}\, \Gamma\left( \nu n + \frac12\right)}\right\}\, {\rm d}t\\
               &= \dfrac2{\sqrt{\pi}} \int_0^1\dfrac1{\sqrt{1-t^2}}\,\mathscr D_\alpha(t)\, {\rm d}t\, ,
   \end{align*}
where the inner sum is evidently the following Dirichlet series
   \[ \mathscr D_\alpha(t) = \sum_{n \ge 1} \dfrac{\alpha_n\, \exp\left\{-n\left(xt + \nu \ln \frac2{x(1-t^2)}\right)\right\}}
                  {n^{\mu-\nu n+1}\, \Gamma\left( \nu n + \frac12\right)}\, , \]
and $p(t) = xt + \nu \ln \frac2{x(1-t^2)}>0$ for $x\in (0,2),$ since $p$ is increasing on $(0,1)$. By the Cauchy convergence
test applied to $\mathscr D_\alpha(t)$ we deduce that
   \[ \left( \frac{{\rm e}x}{2\nu}\,(1-t^2)\right)^\nu {\rm e}^{-xt}\, \limsup_{n \to \infty} |\alpha_n|^{1/n}
      \le \left( \frac{{\rm e}x}{2\nu}\right)^\nu\, \limsup_{n \to \infty} |\alpha_n|^{1/n} <1\, , \]
that is, for all $x \in \mathscr I_\alpha$ the series converges absolutely and uniformly. By the Cahen formula \eqref{W6} we have
   \[ \mathscr D_\alpha(t) = \ln {\rm e}^{xt}\left(\dfrac2{x(1-t^2)}\right)^\nu \int_0^\infty \int_0^{[z]}
                 \left( \left( \frac x2(1-t^2)\right)^\nu{\rm e}^{-xt}\right)^z \cdot
                 \mathfrak d_s \dfrac{\alpha(s) s^{\nu s-\mu-1}}{\Gamma\left(\nu s+\frac12\right)}\, {\rm d}z\, {\rm d}s\, .\]
Thus
   \begin{align*}
      \mathfrak K_{\nu, \mu+1}^{\,\,\alpha}(x) &= - \dfrac1{\sqrt{\pi}}\,\int_0^\infty\int_0^{[z]}
                   \mathfrak d_s \dfrac{\alpha(s) s^{\nu s-\mu-1}}{\Gamma\left(\nu s+\frac12\right)}\, \Phi_\nu(z)\,
                   {\rm d}z\, {\rm d}s\, ,
   \end{align*}
where the $t$--integral
   \[ \Phi_\nu(z) = \int_0^1 \dfrac{\ln {\rm e}^{-xt}\left( \frac x2 (1-t^2)\right)^\nu}{\sqrt{1-t^2}}
                    \left( \left( \dfrac x2 (1-t^2)\right)^\nu\, {\rm e}^{-xt}\right)^z\,{\rm d}t \]
has to be evaluated. After indefinite integration, under definite integral, expanding the exponential term into Maclaurin series,
legitimate termwise integration leads to
   \begin{align*}
      \int \Phi_\nu(z)\, {\rm d}z &= \left( \dfrac x2\right)^{\nu z} \int_0^1 (1-t^2)^{\nu z-\frac12}{\rm e}^{-xzt}\, {\rm d}t \\
             &= \frac{\sqrt{\pi}\, \Gamma\left(\nu z+\frac12\right)}{2\,\Gamma(\nu z)}\,\left( \dfrac x2\right)^{\nu z}\,
                \sum_{j=0}^\infty \dfrac{\left(\frac12\right)_{\frac12 j}}{(\nu z)_j}\, \dfrac{(-xz)^j}{j!}\\
             &= \frac{\sqrt{\pi}\, \Gamma\left(\nu z+\frac12\right)}{2\,\Gamma(\nu z)}\,\left( \dfrac x2\right)^{\nu z}\,
                {}_1\Psi_1^\star\left[\left. \begin{array}{c} \left( \frac12, \frac12\right) \\ (\nu z, 1) \end{array} \right|- xz \right]\, .
   \end{align*}
Consequently
   \[ \Phi_\nu(z) = \frac{\sqrt{\pi}}2\, \dfrac{\partial}{\partial z}\,\frac{\Gamma\left(\nu z+\frac12\right)}{\Gamma(\nu z)}\,
                    \left( \dfrac x2\right)^{\nu z}\,
                    {}_1\Psi_1^\star\left[\left. \begin{array}{c} \left( \frac12, \frac12\right) \\
                    (\nu z, 1) \end{array} \right|- xz \right]\,, \]
and thus
   \begin{align*}
       \mathfrak K_{\nu, \mu+1}^{\,\,\alpha}(x) &= - \int_1^\infty \int_0^{[t]} \dfrac{\partial}{\partial t}\,
                   \dfrac{\Gamma(\nu t+\frac12)}{\Gamma(\nu t)}\left( \dfrac x2\right)^{\nu t}\\
            &\qquad \times {}_1\Psi_1^\star\left[\left. \begin{array}{c} \left( \frac12, \frac12\right) \\ (\nu t, 1) \end{array} \right|
                   -xt \right]\cdot \mathfrak d_s \dfrac{\alpha(s)s^{\nu s-\mu-1}}{\Gamma\left(\nu s+\frac12\right)}
                   \,{\rm d}t {\rm d}s\, .
   \end{align*}
The proof is complete.
\end{proof}

Now, our goal is to establish a second order nonhomogeneous ordinary differential equation which particular
solution is the above introduced special kind Kapteyn series \eqref{K1}. Firstly, we introduce the modified
Bessel type differential operator
   \[ M[y] \equiv y'' + \dfrac1x\, y' - \left(1+\dfrac{\nu^2}{x^2}\right)\, y\, ;\]
this operator is associated with the modified Struve differential equation, reads as follows
   \begin{equation} \label{MSE}
      M[y] \equiv  y'' + \dfrac1x\, y' - \left(1+\dfrac{\nu^2}{x^2}\right)\, y =
                     \dfrac{\left(\frac x2\right)^{\nu-1}}{\sqrt{\pi}\,\Gamma(\nu + \frac12)}\, .
   \end{equation}

\begin{theorem} Let $\min(\nu, \mu)>0$. Then for $x \in \mathscr I_\alpha$ the Kapteyn series $\mathfrak K =
\mathfrak K_{\nu, \mu}^{\,\,\alpha}(x)$ is a particular solution of the nonhomogeneous linear second order
ordinary differential equation
   \begin{equation} \label{K2}
      M_\mu^\alpha [\mathfrak K] \equiv \mathfrak K'' + \dfrac1x\, \mathfrak K' - \left(1+\frac{\nu^2}{x^2}\right)\mathfrak K =
           \dfrac1x\,\Xi_{\nu, \mu}^\alpha(x) + \dfrac2{x\sqrt{\pi}}\, \sum_{n \ge 1} \dfrac{\alpha_n(\frac x2)^{\nu n}}
           {\Gamma\left(\nu n+ \frac12\right)n^{\mu-\nu n+1}}\, ,
   \end{equation}
where
   \begin{align*}
      \Xi_{\nu, \mu}^\alpha(x) &= \frac{\rm d}{{\rm d}x}\, \int_1^\infty \int_1^{[t]} \dfrac{\partial}{\partial t}\,
                   \dfrac{\Gamma\left(\nu t+\frac12\right)}{\Gamma(\nu t)}\left(\dfrac x2\right)^{\nu t}\\
           &\qquad \times {}_1\Psi_1^\star\left[\left. \begin{array}{c} \left( \frac12, \frac12\right) \\
                   (\nu t, 1) \end{array} \right| -xt \right]\cdot \mathfrak d_s \dfrac{\alpha(s)s^{\nu s-\mu-1}(s-1)}
                   {\Gamma\left(\nu s+\frac12\right)} \,{\rm d}t {\rm d}s\, .
   \end{align*}
\end{theorem}

\begin{proof} Consider the modified Struve differential equation \eqref{MSE}
   \[ M[y] \equiv y''(x) + \dfrac1x\,y'(x) - \left(1+ \dfrac{\nu^2}{x^2}\right)\,y(x) =
           \dfrac{\left(\frac x2\right)^{\nu-1}}{\sqrt{\pi}\,\Gamma(\nu + \frac12)}\]
which possesses the solution $y(x) = c_1I_\nu(x) + c_2\mathbf L_\nu(x) + c_3K_\nu(x),$ where $K_\nu$ stands for
the modified Bessel function of the second kind of order $\nu$. Being $I_\nu$ and $K_\nu$ independent particular
solutions (the Wronskian $W[I_\nu,K_\nu] = -x^{-1}$) of the homogeneous modified Bessel ordinary differential
equation, which appears on the left side in \eqref{MSE}, the choice $c_3=0$ is legitimate. Thus $y(x) =
I_{\nu n}(x) - \mathbf L_{\nu n}(x)$ is also a particular solution of \eqref{MSE}. Setting $\nu \mapsto \nu n$, we get
   \begin{align*}
      (I_{\nu n}(x) - \mathbf L_{\nu n}(x))''  &+ \dfrac1x\,(I_{\nu n}(x) - \mathbf L_{\nu n}(x))'\\
                   &- \left(1+ \dfrac{\nu^2n^2}{x^2}\right)\,(I_{\nu n}(x) - \mathbf L_{\nu n}(x)) =
                      \dfrac{\left(\frac x2\right)^{\nu n -1}}{\sqrt{\pi}\,\Gamma\left(\nu n + \frac12\right)}\,.
   \end{align*}
Finally, putting $x \mapsto xn$ multiplying the above display with $n^{-\mu}\alpha_n$ and summing up in
$n \in \mathbb N$, we obtain
   \[ M\!\left[ \mathfrak K_{\nu, \mu}^{\,\,\alpha}\right] = M_\mu^\alpha[\mathfrak K] =
      \dfrac1x\, \left( \mathfrak K_{\nu, \mu}^{\,\,\alpha}(x) - \mathfrak K_{\nu, \mu+1}^{\,\,\alpha}(x)\right)' +
      \dfrac2{x\sqrt{\pi}}\, \sum_{n \ge 1} \dfrac{\alpha_n(\frac{xn}2)^{\nu n}}{\Gamma\left(\nu n+ \frac12\right)n^{\mu+1}}\, , \]
where all three right--hand side series converge uniformly inside  $\mathscr I_\alpha$. Applying the result
\eqref{M1} of the previous theorem to the series
   \[ \mathfrak K_{\nu, \mu}^{\,\,\alpha}(x) - \mathfrak K_{\nu, \mu+1}^{\,\,\alpha}(x) =
                   \sum_{n \ge 2} \dfrac{\alpha_n (n-1)}{n^{\mu+1}}\,
                   \left( I_{\nu n}(xn) - \mathbf L_{\nu n}(xn)\right)\, ,\]
the summation begins with 2. So, the current lower integration limit in the Euler-Maclaurin summation formula related to
\eqref{M1} becomes 1. By this we clearify the stated relation \eqref{K2}.
\end{proof}

In the following we concentrate on the summation of Schl\"omilch series
   \begin{align*}
      \mathfrak T_{\nu, \mu}^{I, \mathbf L}(x) &:= \sum_{n \ge 1} \dfrac1{n^\mu}\,\left(I_\nu(nx) - {\mathbf L}_\nu(nx)\right)\\
      \widetilde{\mathfrak T}_{\nu, \mu}^{I, \mathbf L}(x) &:= \sum_{n \ge 1} \dfrac{(-1)^{n-1}}{n^\mu}\left(I_\nu(nx)
                     - {\mathbf L}_\nu(nx)\right)\, .
   \end{align*}
To unify these procedures, we consider the generalized Schl\"omilch series like \eqref{00}
   \begin{equation} \label{SX2}
      \mathfrak S_{\nu, \mu}^{I, \mathbf L}(x) = \sum_{n \ge 1} \dfrac{\alpha_n}{n^\mu}\,\left(I_\nu(xn)-\mathbf L_\nu(xn)\right)\,;
   \end{equation}
obviously $\mathfrak T_{\nu, \mu}^{I, \mathbf L}(x), \widetilde{\mathfrak T}_{\nu, \mu}^{I, \mathbf L}(x)$ are special cases of
$\mathfrak S_{\nu, \mu}^{I, \mathbf L}(x)$. However, bearing in mind the asymptotics in Sonin--Gubler formula \eqref{W0}, we see that
the necessary condition for the convergence of $\mathfrak S_{\nu, \mu}^{I, \mathbf L}(x)$ for a fixed $x>0$ becomes $\alpha_n = o \left( n^{\mu-\nu+1}\right)$ as $n \to \infty.$

\begin{theorem} Let $\min(\nu,\mu,x)>0$ and $\alpha \in {\rm C}^1(\mathbb R_+)$ be monotone increasing, such that
$\alpha\left|_{\mathbb N}\right. = (\alpha_n)_{n \ge 1}$, and that $\sum_{n \ge 1}n^{-\mu+\nu-1}\alpha_n$ converges.
Then $\mathfrak S = \mathfrak S_{\nu, \mu}^{I, \mathbf L}(x)$ is a particular solution of the nonhomogeneous
linear second order ordinary differential equation
   \[ M_\mu^\alpha[\mathfrak S] = M\left[ \mathfrak S_{\nu, \mu}^{I, \mathbf L}\right] =
                      \dfrac1x\, \left(\Upsilon_{\mu+1}^{\,\,\alpha, 1}(x)\right)'
                    - \dfrac{\nu^2}{x^2}\,\Upsilon_{\mu+2}^{\,\,\alpha, 2}(x) + \dfrac{(\frac x2)^{\nu-1}}
                      {\sqrt{\pi}\,\Gamma\left(\nu + \frac12\right)}\, \sum_{n \ge 1} \dfrac{\alpha_n}{n^{\mu-\nu+1}}\,,\]
where
   \[ \Upsilon_{\,\,\mu}^{\alpha, \beta}(x) = \mu \int_0^\infty e^{-\mu t}\, \int_1^{[e^t]}
                          \mathfrak d_u\left(\alpha(u)(u^\beta-1)\, (I_\nu(xu)
                        - \mathbf L_\nu(xu)) \right)\,{\rm d}t\,{\rm d}u\, .\]
\end{theorem}

\begin{proof} Consider again the modified Struve differential equation \eqref{MSE}, which possesses the solution
$y(x) = c_1I_\nu(x) + c_2\mathbf L_\nu(x) + c_3K_\nu(x)$, and choose the particular solution associated with
$c_1 = -c_2=1$ and $c_3 = 0$. Transforming \eqref{MSE} putting $x \mapsto xn$, multiplying it by $\alpha_nn^{-\mu}$
and summing up the equation with respect $n \in \mathbb N$, we arrive at
   \begin{align*}
      \left( \sum_{n \ge 1} \dfrac{\alpha_n}{n^\mu} y(xn)\right)'' &+ \dfrac1x\,\left(\sum_{n \ge 1}
             \dfrac{\alpha_n}{n^{\mu+1}} y(xn)\right)' - \sum_{n \ge 1} \dfrac{\alpha_n}{n^\mu} y(xn)\\
          &- \dfrac{\nu^2}{x^2}\,\sum_{n \ge 1} \dfrac{\alpha_n}{n^{\mu+2}} y(xn)
           = \dfrac{\left(\dfrac x2\right)^{\nu-1}}{\sqrt{\pi}\,\Gamma\left(\nu+\frac12\right)}
             \sum_{n \ge 1} \dfrac{\alpha_n}{n^{\mu-\nu+1}}\, .
   \end{align*}
Thus
   \begin{align*}
      M\left[ \mathfrak S_{\nu, \mu}^{I, \mathbf L}\right] &=
              \dfrac1x\,\left( \sum_{n \ge 2}\dfrac{\alpha_n(n-1)}{n^{\mu+1}}\, y(xn) \right)' \\
            &\qquad - \dfrac{\nu^2}{x^2}\, \sum_{n \ge 2}\dfrac{\alpha_n(n^2-1)}{n^{\mu+2}}\, y(xn)
            + \dfrac{\left(\dfrac x2\right)^{\nu-1}}{\sqrt{\pi}\,\Gamma\left(\nu+\frac12\right)}
             \sum_{n \ge 1} \dfrac{\alpha_n}{n^{\mu-\nu+1}}\, .
   \end{align*}
Denote
   \begin{equation} \label{UPS}
      \Upsilon_{\,\,\mu}^{\alpha, \beta}(x) = \sum_{n \ge 2}\dfrac{\alpha_n(n^\beta-1)}{n^\mu}\,y(xn),
                    \qquad 0<\nu \le \mu, x>0\, .
   \end{equation}
Following the same lines of the proof of Theorem 7, by Cahen's formula and the Euler--Maclaurin summation we immediately yield the
double definite integral representation
   \[ \Upsilon_{\,\,\mu}^{\alpha, \beta}(x) = \mu \int_0^\infty e^{-\mu t}\, \int_1^{[e^t]}
                               \mathfrak d_u\left(\alpha(u)(u^\beta -1)\,y(xu)\right)\,{\rm d}t\,{\rm d}u\, ;\]
which immediatelly lead to the stated result.
\end{proof}

\begin{corollary} Let $\mu-1>\nu>0$ and $x>0$. Then $\mathfrak T = \mathfrak T_{\nu, \mu}^{I, \mathbf L}(x)$ is a particular solution
of the nonhomogeneous linear second order ordinary differential equation
   \begin{equation} \label{L3}
      M[\mathfrak T] = \dfrac1x\,\left(\Upsilon_{\mu+1}^{\,\,1,1}(x)\right)' - \dfrac{\nu^2}{x^2}\,\Upsilon_{\mu+2}^{\,\,1,2}(x)
                    + \dfrac{\zeta(\mu-\nu+1)}{\sqrt{\pi}\,\Gamma\left(\nu + \frac12\right)}\,\left(\frac x2\right)^{\nu-1} ,
   \end{equation}
where
   \[ \Upsilon_{\,\,\mu}^{1, \beta}(x) = \mu \int_0^\infty e^{-\mu t}\, \int_1^{[e^t]}
                          \mathfrak d_u\left((u^\beta-1)\, (I_\nu(xu)
                        - \mathbf L_\nu(xu)) \right)\,{\rm d}t\,{\rm d}u\, .\]

\end{corollary}

\begin{corollary} Let $\mu>\nu>0$ and $ x>0$. Then $\widetilde{\mathfrak T} = \widetilde{\mathfrak T}_{\nu, \mu}^{I, \mathbf L}(x)$ is
a particular solution of the nonhomogeneous linear second order ordinary differential equation
   \[ M_\mu^\alpha\left[ \widetilde{\mathfrak T}\right] = M\left[ \widetilde{\mathfrak T}_{\nu, \mu}^{I, \mathbf L}\right] =
                      \dfrac1x\,\left( \widetilde{\Upsilon}_{\mu+1}^{\,\,\,1}(x)\right)'
                    - \dfrac{\nu^2}{x^2}\, \widetilde{\Upsilon}_{\mu+2}^{\,\,\,2}(x)
                    + \dfrac{\eta(\mu-\nu+1)}{\sqrt{\pi}\,\Gamma\left(\nu + \frac12\right)}\,
                      \left(\frac x2\right)^{\nu-1} ,\]
where
   \[ \widetilde{\Upsilon}_\mu^{\beta}(x) = \mu \int_0^\infty e^{-\mu t}\, \int_1^{[e^t]}
                                \mathfrak d_u\left({\rm e}^{{\rm i}\pi u}\,(u^\beta-1) (\mathbf L_\nu(xu)
                              - I_\nu(xu)) \right)\,{\rm d}t\,{\rm d}u\, .\]
\end{corollary}

Now, a completely different type integral representation formula will be derived for $\mathfrak T_{\nu, \nu+1}^{I, \mathbf L}(x)$
which simplifies the nonhomogeneous part of related differential equation \eqref{L3}.

\begin{theorem} If $\nu>0$ and $x>0,$ then we have
   \[ \mathfrak T_{\nu, \nu+1}^{I, \mathbf L}(x) = \int_0^\infty J_\nu(xt)\,\left(\coth(\pi t)
                                                 - \dfrac1{\pi t}\right)\, \dfrac{{\rm d}t}{t^{\nu+1}}\, .\]
\end{theorem}

\begin{proof}
Consider the well--known summation formula \cite{hamb}
   \[ \sum_{n \ge 1} \dfrac1{a^2+n^2} = \dfrac\pi{2a}\,\coth(\pi a) - \dfrac1{2a^2}, \qquad a \neq {\rm i}n\, .\]
In conjunction with the Sonin--Gubler formula \eqref{W6} we conclude that
   \[ \mathfrak T_{\nu, \nu+1}^{I, \mathbf L}(x) = \dfrac2\pi\, \int_0^\infty J_\nu(xt) \left(\sum_{n \ge 1} \dfrac1{t^2+n^2}\right)
                   \dfrac{{\rm d}t}{t^\nu} = \int_0^\infty J_\nu(xt)\,\left(\dfrac1t\, \coth(\pi t) - \dfrac1{\pi t^2}\right)\,
                   \dfrac{{\rm d}t}{t^\nu}\, .\]
\end{proof}

\begin{remark} {\rm Actually, the formula
   \[ \sum_{n \ge 1} \dfrac1{a^2+n^2} = \dfrac\pi{2a}\,\coth(\pi a) - \dfrac1{2a^2}, \qquad a \neq 0 \]
has been considered by Hamburger \cite[p. 130, (C)]{hamb} in the slightly different form
   \[ 1+2\sum_{n=1}^\infty {\rm e}^{-2\pi na} = {\rm i} \cot \pi{\rm i}a =
      \dfrac1{\pi a}+ \dfrac{2a}\pi\,\sum_{n=1}^\infty \dfrac1{a^2+n^2}, \qquad a \neq {\rm i}n\, ;\tag{C} \]
Hamburger proved that the functional equation for the Riemann--zeta function is equivalent to (C), see also \cite{butzer1} for connections the above formulae to Eisenstein series. Also, it is worth to mention that further, complex analytical
generalizations of above formula can be found in \cite{berndt}.}
\end{remark}

\section{\bf Novel Bromwich--Wagner line integral representation for $J_\nu(x)$}

As a by--product of Theorem 12, we arrive at the integral relation
  \[ \mathfrak T_{\nu, \nu+1}^{I, \mathbf L}(x) = \sum_{n \ge 1} \dfrac1{n^{\nu+1}}\left(I_\nu(xn)
                     - \mathbf L_\nu(xn)\right), \qquad \nu>0, x>0\]
which, in the expanded form reads
  \[ \int_0^\infty J_\nu(xt)\,\left( \coth(\pi t) - \dfrac1{\pi t}\right)\, \dfrac{{\rm d}t}{t^{\nu+1}} =
        (\nu+1) \int_0^\infty \int_0^{[e^s]} e^{-(\nu+1)s}\,\mathfrak
d_u\left(I_\nu(xu)- \mathbf L_\nu(xu) \right)\,{\rm d}s\,{\rm d}u\,
.\]
This arises in a Fredholm type convolutional integral equation of the
first kind with degenerate kernel
  \begin{equation} \label{Y1}
      \int_0^\infty f(xt)\,\left( \coth(\pi t) - \dfrac1{\pi t}\right)\,
\dfrac{{\rm d}t}{t^{\nu+1}} = F_\nu(x)\, ,
  \end{equation}
having nonhomogeneous part
  \begin{equation} \label{Y2}
      F_\nu(x) = (\nu+1) \int_0^\infty \int_0^{[e^s]}
e^{-(\nu+1)s}\,\mathfrak d_u\left(I_\nu(xu)- \mathbf L_\nu(xu) \right)\,{\rm d}s\,{\rm d}u\, .
  \end{equation}
Obviously, $J_\nu$ is a particular solution of this equation. Before we
state our result, we say that the functions
$f$ and $g$ are {\em orthogonal a.e.} with respect to the ordinary Lebesgue
measure on the positive half-line when
$\int_0^\infty f(x)g(x){\rm d}x$ vanishes, writing this as $f\perp g$.

\begin{theorem}
Let $\nu>0$ and $x>0$. The first kind Fredholm type convolutional integral
equation with
degenerate kernel \eqref{Y1} possesses particular solution $f = J_{\nu} +
h$, where $h \in L^1(\mathbb R_+)$ and
    \[ h(x) \perp x^{-\nu-1}\,\left( \coth(\pi x) - \dfrac1{\pi x}\right), \qquad x>0 \]
if and only if  the nonhomogeneous part of the integral equation equals $F_\nu(x)$
given by \eqref{Y2}.
\end{theorem}

We mention that $h$ as in the above theorem has been constructed in \cite[Example]{DP}. To solve the
integral equation \eqref{Y2} we use the Mellin integral
transform technique, following some lines of
a similar procedure used by Dra\v s\v ci\'c--Pog\'any in \cite{DP}. The
Mellin transform pair of certain
suitable $f$ we define \cite{span}
  \[ \mathscr{M}_p(f):= \int_0^\infty x^{p-1} f(x)\,{\rm d}x, \qquad
\mathscr{M}^{-1}_x(g) :=
      \frac1{2\pi{\rm i}}\int_{c-{\rm i}\infty}^{c+{\rm i}\infty} x^{-p}
\mathscr{M}_p(f)\,{\rm d}p\, , \]
where the inverse Mellin transform is given in the form of a line integral
with Bromwich--Wagner type integration path
which begins at $c- {\rm i}\infty$ and terminates at $c+{\rm i}\infty$.
Here the real $c$ belongs to the fundamental
strip of the inverse Mellin transform $\mathscr M^{-1}$.

\begin{theorem} Let $\nu>0, x>0$. Then the following Bromwich--Wagner type
line integral representation holds true
  \begin{equation} \label{Y6}
      J_\nu(x) = \dfrac{\nu+1}{2{\rm i}} \int_{c - {\rm i}\infty}^{c + {\rm i}\infty}
                  \dfrac{\mathscr M_p\left(\displaystyle \int_0^\infty \int_0^{[e^s]}e^{-(\nu+1)s}\,\mathfrak d_u\left(I_\nu(xu)
                - \mathbf L_\nu(xu) \right)\,{\rm d}s\,{\rm d}u\right)}
                  {\Gamma\left(\frac{p-\nu}2\right)\, \Gamma\left(\frac{\nu-p}2+1\right)\,
                  \zeta(\nu-p+2)}\, x^{p-1}\, {\rm d}p\, ,
  \end{equation}
where $c \in (\nu, \nu+1)$.
\end{theorem}

\begin{proof}
Applying $\mathscr M^{-1}$ to the equation \eqref{Y1},  we get
   \[ \mathscr M_p\left(\int_0^\infty J_\nu(xt)\,\Big( \coth(\pi t) - \dfrac1{\pi t}\Big)\, \dfrac{{\rm d}t}{t^{\nu+1}}\right) =
      \mathscr M_p(F_\nu)\,.\]
By the Mellin--convolution property
   \[ \mathscr M_p\left( f \star g\right) =
      \mathscr M_p\left(\int_0^\infty f(rt) \cdot g(t)\, {\rm d}t\right) = \mathscr M_p(f) \cdot \mathscr M_{1-p}(g)\, ,\]
it follows that
   \begin{equation} \label{Y8}
      \mathscr M_p \left( x^{-\nu-1}\big(\coth\pi x - (\pi x)^{-1})\right)\cdot \mathscr M_{1-p}(J_\nu) = \mathscr M_p(F_\nu)\,.
   \end{equation}
The fundamental analytic strip contains $(\nu, 1+\nu)$, because the $\coth$ behaves like
   \[ \coth z = \begin{cases}
                      \dfrac1{z} + \dfrac z3 + \mathcal O(z^3), & \qquad \mbox{if}\ \ z \to 0\\
                      1 + 2{\rm e}^{-2z}\left[1 + \mathcal O\left({\rm e}^{-2z}\right)\right] & \qquad \mbox{if}\ \ z\to \infty
                   \end{cases} \, ;\]
in both cases $\Re(z)>0$. Now, rewriting $x^{-\nu-1}\left(\coth\pi x - (\pi x)^{-1}\right)$ by (C) and using termwise
the Beta--function description, we conclude that
   \[ \mathscr M_p \left( x^{-\nu-1}\left(\coth\pi x - (\pi x)^{-1}\right)\right) = \frac1\pi\,
              {\rm B}\left( \dfrac{p-\nu}2, \dfrac{\nu-p}2+1\right)\, \zeta(\nu-p+2)\, , \]
for all $p \in (\nu, \nu+1)$. Therefore
   \[ \mathscr M_{1-p}(J_\nu) = \dfrac{\pi \mathscr M_p(F_\nu)}{{\rm B}\left( \frac{p-\nu}2, \frac{\nu-p}2+1\right)\, \zeta(\nu-p+2)}\, ,\]
which finally results in
   \[ J_\nu(x) = \dfrac{\nu+1}{2{\rm i}} \int_{c - {\rm i}\infty}^{c + {\rm i}\infty}
                  \dfrac{\mathscr M_p\left(\displaystyle \int_0^\infty \int_0^{[e^s]} e^{-(\nu+1)s}\,\mathfrak d_u\left(I_\nu(xu)
                - \mathbf L_\nu(xu) \right)\,{\rm d}s\,{\rm d}u\right)}{{\rm B}\left( \frac{p-\nu}2, \frac{\nu-p}2+1\right)\,
                  \zeta(\nu-p+2)}\, x^{p-1}\, {\rm d}p\, ,\]
where the fundamental strip contains $c = \nu + \frac12$. So, the desired integral representation formula is established.
\end{proof}

We note that the formula--collection \cite{http} does not contain \eqref{Y6}.

\begin{theorem} Let $0<\nu<\frac32, x>0$. Then
  \begin{equation} \label{Y9}
      \mathfrak T_{\nu, \nu+1}^{I,\mathbf L}(x) = \dfrac1{2^{p+1}\,\pi\,{\rm i}}
                   \int_{c-{\rm i}\infty}^{c+{\rm i}\infty}
                   \dfrac{\Gamma\left(\frac{\nu-p}2 + \frac12\right)\,\zeta(\nu-p+2)}
                   {\sin\left[ \frac\pi2(p-\nu)\right]
                   \cdot \Gamma\left( \frac{\nu+p}2 + \frac12\right)}\, x^{-p}\, {\rm d}p \,.
      \end{equation}
Here also $c \in (\nu, \nu+1)$.
\end{theorem}

\begin{proof} Consider relation \eqref{Y8}. Expressing $\mathscr M_{1-p}(J_\nu)$ {\em via} formula \cite[p.93, Eq. 10.1]{ober}
  \[ \mathscr M_p\left(J_\nu(ax)\right) = \dfrac{2^{p-1}}{a^p}\, \dfrac{\Gamma\left( \frac{\nu+p}2\right)}
                 {\Gamma\left(\frac{\nu-p}2+1\right)}, \qquad a>0,\,-\nu<p<\tfrac32\,,\]
equality\eqref{Y8}, by virtue of the Euler's reflection formula becomes
  \begin{align*}
     \mathscr M_p(F_\nu) &= \dfrac{\Gamma\left( \frac{\nu-p}2+1\right)\,\Gamma\left(\frac{\nu-p}2 + \frac12\right)\,
                            \Gamma\left(\frac{p-\nu}2\right)\,\zeta(\nu-p+2)}{2^p\, \pi\,\Gamma\left( \frac{\nu+p}2+\frac12\right)}\\
                         &= \dfrac{\Gamma\left(\frac{\nu-p}2 + \frac12\right)\,\zeta(\nu-p+2)}{2^p\,\sin\left[ \frac\pi2(p-\nu)\right]
                           \cdot \Gamma\left( \frac{\nu+p}2+\frac12\right)}\,  .
  \end{align*}
Having in mind that $F_\nu(x)$ is the integral representation of the
Schl\"omilch series $\mathfrak T_{\nu, \nu+1}^{I,\mathbf L}(x)$,
inverting the last display by $\mathscr M_p^{-1}$ we arrive at the
asserted result.
\end{proof}

\section*{Acknowledgements}
The authors would like to cordially thank anonymous referee for his/her constructive comments and suggestions. The research of \'A. Baricz was supported by the Romanian National Research Council CNCS-UEFISCSU,
project number PN-II-RU-TE\underline{ }190/2013.

\singlespacing

\end{document}